\documentclass{imanum}
\usepackage{graphicx}
\usepackage{epsfig}
\usepackage{latexsym,amsfonts,amsbsy,amssymb}
\usepackage{amsmath,amsthm}
\usepackage{color}

\begin{document}

\title{A stroboscopic averaging algorithm for highly oscillatory delay problems}
\author{{\sc
\sc
J. M. Sanz-Serna\thanks{ Email: jmsanzserna@gmail.com}} \\[2pt]
Departamento de Matem\'{a}ticas, Universidad Carlos III de Madrid,\\ Avenida de la Universidad
30, E-28911 Legan\'{e}s (Madrid), Spain\\[6pt]
{\sc and}\\[6pt]
{\sc
Beibei Zhu\thanks{Corresponding author. Email: zhubeibei@lsec.cc.ac.cn}}\\[6pt]
LSEC, ICMSEC, Academy of Mathematics and Systems Science,\\ Chinese Academy of Sciences, Beijing 100190, China and\\
School of Mathematical Sciences, University of Chinese\\ Academy of Sciences, Beijing 100049, China }

\maketitle

\begin{abstract}{
We propose and analyze a heterogenous
multiscale method  for the efficient
integration of constant-delay differential
equations  subject to fast periodic forcing.
The stroboscopic averaging method (SAM)
suggested here may provide approximations with
\(\mathcal{O}(H^2+1/\Omega^2)\) errors with a
computational effort that grows like \(H^{-1}\)
(the inverse of the stepsize), uniformly in the
forcing frequency \(\Omega\). } {Delay
differential equations, stroboscopic averaging,
highly oscillatory problems }
\end{abstract}

\section{Introduction}
We propose and analyze a heterogenous multiscale method (HMM) \citep{E2003, EEngquist} for the efficient integration of   constant-delay differential equations  subject to fast periodic forcing. The stroboscopic averaging method (SAM) suggested here may provide approximations with \(\mathcal{O}(H^2+1/\Omega^2)\) errors with a computational effort that grows like \(H^{-1}\) (the inverse of the stepsize), uniformly in the forcing frequency \(\Omega\).

Delay differential equations with fast periodic forcing appear in a number of recent contributions to the nonlinear Physics literature.  As shown by \cite{Daza3}, under fast periodic forcing, the Delayed Action Oscillator (\cite{Boutle}) that describes El Ni\~{n}o phenomenon may generate basins of attraction with the Wada property, i.e.\ each point on the boundary of one of the basins is actually on the boundary of all basins. The system
\begin{eqnarray}\label{toggle}
\dot{x}_1(t)&=&\frac{\alpha}{1+x_2^{\beta}(t)}-x_1(t-\tau)+A\sin(\omega t)+B\sin(\Omega t),\\
\dot{x}_2(t)&=&\frac{\alpha}{1+x_1^{\beta}(t)}-x_2(t-\tau), \nonumber
\end{eqnarray}
 ($\tau$, $\alpha,\beta,\omega,A,B$ are constants) describes,  for \( A = 0\), \( B= 0\) a time-delayed genetic toggle switch,
 a synthetic gene-regulatory network \citep{Gardner}. Studied by \cite{Daza} is the phenomenon of {\em vibrational resonance} \citep{Landa}, i.e.\ the way in which  the presence of the high-frequency forcing \(B\sin(\Omega t)\) enhances the response of the system to the low-frequency forcing \(A\sin(\omega t)\).
\cite{Jee} and \cite{Daza2} investigate in a similar way  the forced, delayed Duffing oscillator. A new kind
of nonlinear resonance of periodically forced delay systems has recently been described by \cite{Coccolo}.

The numerical integration of highly oscillatory differential equations  with or without delay may be a very demanding task, as standard  methods typically have to employ timesteps smaller
 than the periods present in the solution. For systems without delay, the literature contains many suggestions
 of numerical schemes specially designed to operate in the highly-oscillatory scenario; many of them are reviewed
 in \cite{HLW}. Perhaps counterintuitively, some of those methodologies take advantage of the large frequency
 and their efficiency actually increases with \(\Omega\)
\citep{arieh}. On the other hand, schemes for
highly oscillatory problems may suffer from
unexpected instabilities and inaccuracies
\citep{CS1}.

The algorithm suggested here is based on ideas presented, for systems without delay, by \cite{CCMS1} and
\cite{CCMS2}. In these references, given an oscillatory problem, a {\em stroboscopically averaged} problem is
introduced such that, at the stroboscopic times \(t^{(k)}=kT\), \(T=2\pi/\Omega\), \( k = 0, 1,\dots\), its
solution \(X(t)\) (approximately)
 coincides with the true oscillatory
 solution \(x\). The stroboscopically averaged problem does not
 include  rapidly varying forcing terms and therefore, if available in closed form,  may be integrated numerically
  without much ado. The algorithms in \cite{CCMS1} and \cite{CCMS2} compute numerically values of \(X\), without
   demanding that the user finds analytically the expression of the averaged system. More precisely, the algorithms
    only require evaluations of the right-hand side of the originally given oscillatory problem. The solution \(X\)
     is advanced  with a standard integrator (the macro-integrator) with a stepsize
\(H\) that, for a target accuracy, may be chosen to be independent of \(\Omega\). When
 the macro-integrator requires a value \(F\) of the slope \(\dot X\), \(F\) is found by numerical
 differentiation of a micro-solution \(u\), i.e.\ a solution of the originally given oscillatory problem. While
  the micro-integrations to find \(u\) are performed with stepsizes \(h\) that are
  submultiples of the (small) period \(T\), the corresponding computational cost does not increase as \(\Omega\rightarrow \infty\), because \(u\) is only required in windows of width \(mT\), \(m\) a small integer.

The extension of the material in \cite{CCMS1} and
\cite{CCMS2} to systems with delay is far
from trivial. A first difficulty stems from the
well-known fact that, in the delay scenario,
regardless of the smoothness of the equation,
solutions may be non-smooth at points \(t\) that
are integer multiples of the (constant) delay.
 Therefore, the
algorithm presented here has  to make
special provision for that lack of smoothness. In
addition, the analysis of the algorithm (but, as
emphasized above, not the algorithm itself) is
built on the knowledge of the stroboscopically
averaged systems. While the construction of a
stroboscopically averaged system with errors
\(x(t^{(k)})- X(t^{(k)})=\mathcal{O}(1/\Omega)\)
is not difficult, here we aim at errors
\(x(t^{(k)})-
X(t^{(k)})=\mathcal{O}(1/\Omega^2)\) and this
requires much additional analysis. The classical
reference \cite{Lehman}  only considers
zero-mean, \(\mathcal{O}(1/\Omega)\) averaging.

In Section~\ref{sec:SAM} we present the main ideas
of the paper and
  a detailed description of
the algorithm.  Due to the difficulties imposed by the lack of smoothness in the solution, the
algorithm uses low-order methods: the second-order Adams-Bashforth formula as a macro-integrator and Euler's
rule as a micro-integrator.  Section~\ref{sec:averaged} contains the construction of the stroboscopically
averaged system with \(\mathcal{O}(1/\Omega^2)\) accuracy. Sections~\ref{sec:error1} and \ref{sec:error2} are
devoted to the analysis of the SAM algorithm. In the first of these, we assume that the micro-integrations
carried out in the algorithm are performed exactly. Under suitable hypotheses, the errors in SAM are
\(\mathcal{O}(H^2+1/\Omega^2)\). The effect of the errors in the micro-integration is studied in
Section~\ref{sec:error2}: it is shown that, with a computational cost that grows like \(1/H\), SAM may yield
errors of size \(\mathcal{O}(H^2+1/\Omega^2)\). The \(H^2\) (second order) behaviour of the error may come as
a surprise,  because micro-integrations are performed by Euler's rule; of key importance here is a
superconvergence result (see the bound in \eqref{eq:auxeulerbis}) for the Euler solution of oscillatory
problems  when the integration is carried out over a whole number of periods. The last two sections report
numerical experiments that,
 on the one hand, confirm the theoretical expectations and, on the other, show the
  advantage of SAM when compared with a direct numerical integration of the oscillatory problem.
  Speed-ups larger than three order of magnitude are reported at the end of Section~\ref{sec:extensions}.

\section{The stroboscopic averaging method (SAM)}
\label{sec:SAM}

This section  motivates and describes the SAM algorithm.

\subsection{Motivation}

We consider highly oscillatory delay differential systems of the form
\begin{eqnarray}\label{gensys}
\dot{x}(t)&=&f(x(t),y(t),t,\Omega t;\Omega),\qquad t\ge 0,\\
y(t)&=&x(t-\tau),\qquad t\ge 0, \nonumber
\end{eqnarray}
where the solution $x$ is defined for $t \ge -\tau$ and takes values in $\mathbb{R}^D$, the function
$f(x,y,t,\theta;\Omega):\mathbb{R}^D\times \mathbb{R}^D\times \mathbb{R}\times \mathbb{R}\times \mathbb{R}
\rightarrow \mathbb{R}^D$ depends $2\pi$--periodically on its fourth argument $\theta$, $\tau>0 $ is the
(constant)  delay, and the frequency $\Omega$ is seen as a parameter, $\Omega \gg 1$. Note that \(f\) depends
explicitly on \(t\) through its third and fourth arguments; the fourth is the {\em rapidly rotating phase}
\(\theta =\Omega t\) and the third corresponds to a {\em slow} (i.e.\ \(\Omega\)-independent) dependence on
\(t\) (see the toggle-switch equations above). The values of $x$ on the interval $[-\tau,0]$ are prescribed
through an {\em $\Omega$--independent}\footnote{Assuming that \(\varphi\) does not depend on \(\Omega\)
implies no loss of generality, as the general case may be reduced to the \(\Omega\)-independent case by
introducing a new dependent variable \(x(t)-\Phi(t)\), where \(\Phi(t)\) coincides with \(\varphi(t)\) for
\(-\tau\leq t\leq 0\). }
 function $\varphi$:
\begin{equation}\label{eq:varphi}
 x(t)=\varphi(t),\qquad-\tau\le t\le 0.
\end{equation}

It is well known that, regardless of the smoothness of \( f\) and \(\varphi\), the function \(x(t)\) will typically not be differentiable at \(t=0\) and that in \eqref{gensys} \( \dot{x}(0)\)
has to be understood as a right derivative. Furthermore the discontinuity of \(\dot{x}(t)\) at \( t = 0\)  will lead to the discontinuity of \( \ddot{x}(t)\) at $t=\tau$, etc.

We assume that,  at the {\em stroboscopic times} $t^{(k)}=kT$,
where \( T =2\pi/\Omega\) is the period and $k = 0, 1,\dots$,
the solution \(x(t)\) of the oscillatory delay problem \eqref{gensys}--\eqref{eq:varphi} may be approximated (in a sense to be made precise later) by the solution \(X(t)\) of an averaged problem
\begin{eqnarray}\label{genaver}
\dot{X}(t)&=&F,\qquad t\ge 0,\\
X(t)&=&\varphi(t),\qquad -\tau\le t\le 0,\nonumber
\end{eqnarray}
where the value of the function \(F\) may depend on \(X(t)\), on the history \(X(s)\), \(-\tau\leq s<t\),
on the slow time \(t\) and on \(\Omega\) but is {\em independent} of the fastly varying phase \(\theta=\Omega t\).
\begin{figure}[t!]
\vspace{-5cm}\centering\includegraphics[scale=0.45]{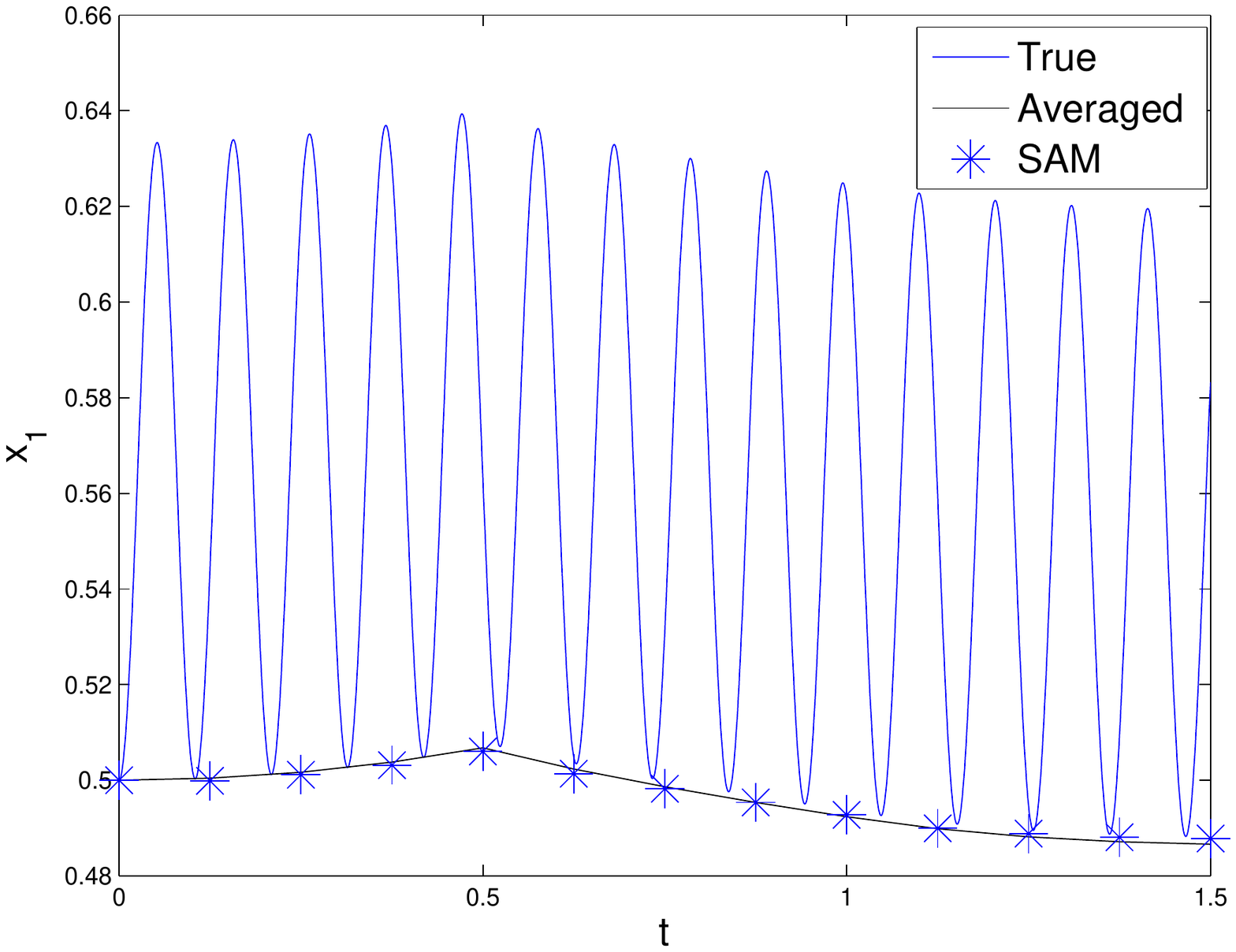}
\vspace{-3cm}
\caption{ $x_1$ component of the toggle switch problem. The constants $\tau$, $\alpha$, $\beta$, $A$, $\omega$, $B$,
 and the function \(\varphi\) are chosen as in Section~\ref{sec:numer} and $\Omega=60$.
The true solution \(x\)  and the averaged solution \(X\) are very close at the stroboscopic times \( t^{(k)} = k(2\pi/\Omega)=kT\) , \( k = 0, 1,\dots\) SAM is used to generate approximations to \( X\) at the step points \(t_n=nH\), \(H=0.125\), \( n=0, 1,\dots\) The value of \(H\) is chosen so that the point \(t=\tau\), where the slope of the solution is discontinuous, is a step point. On the other hand, with the value of \(\Omega\) considered, the step points (abscissae of the stars) are {\em not} stroboscopic times.}
\label{fig:toggle}
\end{figure}

 We illustrate the situation in the particular case of the toggle-switch problem \eqref{toggle}. Fig.~\ref{fig:toggle} displays, in  a short time interval, a solution of the given oscillatory system and the corresponding averaged solution, found by solving the stroboscopically averaged system
\begin{eqnarray}\label{vraver}
\dot{X}_1(t)&=&\frac{\alpha}{1+X_2^{\beta}(t)}-Y_1(t)-\frac{B}{\Omega}1_{\{t\geq \tau\}}(t)+A\sin(\omega t),\\ \nonumber
\dot{X}_2(t)&=&\frac{\alpha}{1+X_1^{\beta}(t)}-Y_2(t)-\frac{B}{\Omega}\frac{\alpha \beta X_1^{\beta-1}(t)}{(1+X_1^{\beta}(t))^2},
\end{eqnarray}
where \( Y_i(t) = X_i(t-\tau)\) and  \(1_{\{t\geq \tau\}}(t)\) is the (indicator) function that takes the
value 1 for \(t\geq \tau\) and vanishes for \(t<\tau\) (see the next section for the derivation of this
averaged system). Note that the slow time-dependent forcing \(A\sin(\omega t)\) has not been averaged out and
that, due to the presence of \(1_{\{t\geq \tau\}}(t)\),  the right-hand side of the averaged system is
discontinuous at \( t=\tau\) (this discontinuity manifests itself  in Fig.~\ref{fig:toggle} through a
discontinuity in the slope of \(X\),
not to be confused with the discontinuity  at \(t=0\)
typically present, as mentioned above, in solutions of delay problems). The notion of stroboscopic averaging
studied in detail in  \cite{Murua}
 is far from new. However standard treatments of the theory of averaging favour
 alternative techniques, specially the zero-mean approach where the functions
 necessary to express the required change of variables are chosen so as to have zero mean over
 one period of the oscillation. In stroboscopic averaging the  freedom available in the choice of change of variables is used
 to
 impose that  the old and new variables coincide at stroboscopic times. This is advantageous for the numerical
 methods studied here. The zero-mean approach may be better for analytic purposes
 at it usually leads to simpler high-order averaged systems. If zero-mean averaging
 had been used in Fig.~\ref{fig:toggle}, the averaged solution would had been located
 halfway between the maxima and the minima of the oscillatory solution.

The study of the vibrational resonance of \eqref{toggle} requires to simulate over long time intervals (the interval \(0\leq t\leq 300\) is used in \cite{Daza}) for many choices of the values
of the constants $\alpha,\beta,\omega,A,B,\tau$ and the parameter \(\Omega\); the presence of the fast-frequency oscillations makes such a task costly.
It is then of interest to simulate, if possible,  averaged systems like \eqref{genaver} rather than  highly-oscillatory models like \eqref{gensys}. However, obtaining
$F$ analytically may be difficult or even impossible and we wish to have a numerical method that approximates $X$ by using only $f$. SAM is such a technique.

The idea behind SAM is as follows. Let \(x\) and \(X\) be respectively the oscillatory and averaged solutions corresponding to the same \(\varphi\). At fixed \(t\),   \(\dot x\) and \(\dot X\) may differ substantially. However  difference quotients such as
\((x(t+T)-x(t))/T\) or \((x(t+T)-x(t-T))/(2T)\) may provide a good approximation to the slope \(\dot X(t)\) (see Fig.~\ref{fig:toggle}).
As other heterogeneous multiscale methods (see e.g.\ \cite{E2003,EEngquist,E2007,Engquist,Li, Ariel,Sanz-Sernamod,CS2}), the
algorithm includes macro-integrations and micro-integrations. Macro-integations are used to advance \(X\) over macro-steps
of length \(H\) larger than the period \(T\). The necessary slopes \(\dot X\) are obtained by forming difference quotients of auxiliary oscillatory solutions found by micro-integrations with small steps \(h\).

\subsection{The algorithm}

Let us now describe the algorithm.

 \subsubsection{Macro-integration}  Choose a positive integer $N$  and define the macro-stepsize  $H=\tau/N$. If the solution is sought in an interval \( 0\leq t\leq t_{max}\), SAM generates approximations $X_n$ to $X(t_n)$, $t_n=nH$,
 \(n=0,1,\dots, \lfloor t_{max}/H \rfloor\)
 by using the second-order Adams-Bashforth formula (macro-integrator)
\begin{equation}\label{ab2}
X_{n+1}=X_n+\frac{3}{2}HF_n-\frac{1}{2}HF_{n-1},
\end{equation}
starting from $X_0=x(0)=X(0)=\varphi(0)$; here $F_n$ is an approximation to $\dot{X}(t_n)$ obtained by numerical differentiation of the micro-solution. The formula is {\em not} used
if $n= 0$ and $n= N$, where it would be inconsistent in view of the jump discontinuities of \(\dot X\) at \( t=0\) and \(t=\tau\) noted above.
For $n=0$ and $n=N$ we use Euler's rule, i.e.\ we set
\begin{equation}\label{eq:Eulermacro}
X_1=X_0+HF_0,\qquad X_{N+1}=X_N+HF_N.
\end{equation}

\subsubsection{Micro-integration}

If $\nu_{max}$ is a positive integer, the micro-stepsize $h$ is chosen to be $T/\nu_{max}$ (recall that \(T=2\pi/\Omega\) denotes the period). We use  Euler's rule, starting from $u_{n,0}=X_n$, first to integrate forward
the oscillatory problem (\ref{gensys})  over one period
\begin{equation}\label{forward}
u_{n,\nu+1}=u_{n,\nu}+hf(u_{n,\nu},v_{n,\nu},t_n+\nu h,\Omega\nu h;\Omega),~~~\nu=0,1,\dots,\nu_{max}-1,
\end{equation}
and then to integrate backward over one period
\begin{equation}\label{backward}
u_{n,-\nu-1}=u_{n,-\nu}-hf(u_{n,-\nu},v_{n,-\nu},t_n-\nu h,-\Omega\nu h;\Omega),~~~\nu=0,1,\dots,\nu_{max}-1.
\end{equation}
 Here $v$ denotes the past values  for \(u\) given by \(v_{n,\nu} =u_{n-N,\nu}\) if \(n>N\), and $v_{n,\nu}=\varphi(-\tau+nH+\nu h)$ for $n<N$ (if \(n = N\) \(v_{N,\nu} =u_{0,\nu}\) for \(\nu\geq 0\) and
 $v_{N,\nu}=\varphi(\nu h)$ for \(\nu<0\)).

It is {\em crucial} to observe the values \(\Omega\nu h\) and
\(-\Omega\nu h\) used for the fast argument \(\theta\) in \eqref{forward} and
\eqref{backward} respectively;  each micro-integration  starts from \(\theta = 0\) rather than
from the value \(\theta = \Omega t_n\) which may perhaps have been expected.
The reason for this is that in stroboscopic averaging, the resulting averaged system changes with the initial
value of the phase \(\theta\); to work with one and the same averaged system micro-integrations have to start from the
value \(\theta= 0\) that the phase takes at the initial point of the interval \([0,t_{max}]\) (see
\cite{CCMS1}, \cite{CCMS2} for a detailed discussion).

 The slopes $F_n$ to be used in \eqref{ab2} or \eqref{eq:Eulermacro} are given by the central difference formula,
\begin{equation}\label{eq:slope1}
F_n=\frac{u_{n,\nu_{max}}-u_{n,-\nu_{max}}}{2T}
\end{equation}
if $n\neq 0$ and $n\neq N$, while for $n=0$ and $n=N$, we use the forward difference formula
\begin{equation}\label{eq:slope2}
F_0=\frac{u_{0,\nu_{max}}-u_{0,0}}{T},\qquad
F_N=\frac{u_{N,\nu_{max}}-u_{N,0}}{T},
\end{equation}
due to the discontinuity of \(\dot{X}(t)\) at \(t=0\) and \(t=\tau\).
A detailed description  of the algorithm is provided  in Table \ref{tab:algorithm}.

Fig.~\ref{figschematic} may help to better understand the procedure. The upper time axis corresponds to the macro-integration; all the information needed to obtain $X_{n+1}$ is $X_n$, $F_n$ and $F_{n-1}$ (\(F_{n-1}\) is actually not required for \(n=0,N\)).
The value of $F_n$ is derived by numerical differentiation of the micro-solution and passed to the macro-integrator to compute \(X_{n+1}\). For fixed \(n\), the computation of the Euler micro-solution
\( u_{n,\nu}\) uses   the past values \(v_{n,\nu}\) and the initial datum \(u_{n,0} = X_n\) (note that \(X_n\) is the most recent vector found in the macro-integration).

SAM only operates with macro-stepsizes \(H\) that are submultiples of the delay \(\tau\); this restriction is
imposed to enforce that \(t=\tau\) be a step-point to better deal with the discontinuity in slope at
\(t=\tau\) (see Fig.~\ref{fig:toggle}). In general the quotient \(\tau/T\) will not be a whole number and then
the step points \(t_n\) will not be stroboscopic times; this is the case in the simulation in
Fig.~\ref{fig:toggle}.

We  emphasize that, given \(N\) and \(\nu_{max}\), the {\em complexity} of the algorithm is {\em independent} of \(\Omega\). When \(\Omega\) increases, the micro-stepsize \(h\) decreases to cater for the more rapid variation of the oscillations but  the window of width \(2T\) (or \(T\)) for each micro-integration becomes correspondingly narrower.

Finally we point out that we have tested several
alternative algorithms. For instance we
alternatively performed the micro-integrations
with the Adams-Bashforth second order method, or
we used second-order forward approximations for
\(F_0\), \(F_N\). While those modifications
improve the accuracy of the results for small
stepsizes, experiments reveal that they are not
always beneficial for large stepsizes; therefore
we shall not be concerned with them here.

\begin{table}
\caption{SAM Algorithm}
\vspace{-2mm}
\begin{center}
\begin{tabular}{lcccc}
\hline
\texttt{Compute $X_1$}\\
~~~~ \texttt{Micro-integration}\\
~~~~~~~~~~~$u_{0,0} = \varphi(0)$ \% \texttt{initial value}\\
~~~~~~~~~~~For $\nu=0:\nu_{max}$, $v_{0,\nu}=\varphi(-\tau+\nu h)$, end \% \texttt{history}\\
~~~~~~~~~~~For $\nu=0:\nu_{max}-1$, $u_{0,\nu+1}=u_{0,\nu}+hf(u_{0,\nu},v_{0,\nu},t_0+\nu h,\Omega\nu h;\Omega)$, end \% \texttt{Euler}\\
~~~~~~~~~~~$F_0=(u_{0,\nu_{max}}-u_{0,0})/T$ \% \texttt{slope for macro-step}\\
~~~~ \texttt{Macro-integration}\\
~~~~~~~~~~~$t_0 = 0$, $X_0=\varphi(0)$, $t_1=t_0+H$, $X_1=X_0+HF_0$ \% \texttt{Euler macro at n=0}\\
\texttt{Compute $X_2,\dots,X_N$}\\
For $n=1:N-1$\\
~~~~ \texttt{Micro-integration}\\
~~~~~~~~~~~$u_{n,0}=X_n$ \% \texttt{initial value}\\
~~~~~~~~~~~For $\nu=-\nu_{max}:\nu_{max}$, $v_{n,\nu}=\varphi(-\tau+nH+\nu h)$, end \% \texttt{history}\\
~~~~~~~~~~~For $\nu=0:\nu_{max}-1$, $u_{n,\nu+1}=u_{n,\nu}+hf(u_{n,\nu},v_{n,\nu},t_n+\nu h,\Omega\nu h;\Omega)$\\
~~~~~~~~~~~$u_{n,-\nu-1}=u_{n,-\nu}-hf(u_{n,-\nu},v_{n,-\nu},t_n-\nu h,-\Omega\nu h;\Omega)$, end \% \texttt{Euler}\\
~~~~~~~~~~~$F_n=(u_{n,\nu_{max}}-u_{n,-\nu_{max}})/(2T)$ \% \texttt{slope for macro-step}\\
~~~~ \texttt{Macro-integration}\\
~~~~~~~~~~~$t_{n+1}=t_n+H$, $X_{n+1}=X_n+(3H/2)F_n-(H/2)F_{n-1}$ \% \texttt{Adams-Bashforth2 macro}\\
end\\
\texttt{Compute $X_{N+1}$}\\
~~~~ \texttt{Micro-integration}\\
~~~~~~~~~~~$u_{N,0}=X_N$ \% \texttt{initial value}\\
~~~~~~~~~~~For $\nu=1:\nu_{max}$, $u_{0,-\nu}=\varphi(-\nu h)$, end \% \texttt{values taken from history}\\
~~~~~~~~~~~For $\nu=-\nu_{max}:\nu_{max}$, $v_{N,\nu}=u_{0,\nu}$, end \% \texttt{save for later history}\\
~~~~~~~~~~~For $\nu=0:\nu_{max}-1$, $u_{N,\nu+1}=u_{N,\nu}+hf(u_{N,\nu},v_{N,\nu},t_N+\nu h,\Omega\nu h;\Omega)$\\
~~~~~~~~~~~$u_{N,-\nu-1}=u_{N,-\nu}-hf(u_{N,-\nu},v_{N,-\nu},t_N-\nu h,-\Omega\nu h;\Omega)$, end \% \texttt{Euler}\\
~~~~~~~~~~~$F_N=(u_{N,\nu_{max}}-u_{N,0})/T$ \% \texttt{slope for macro-step}\\
~~~~ \texttt{Macro-integration}\\
~~~~~~~~~~~$t_{N+1}=t_N+H$, $X_{N+1}=X_N+HF_N$ \% \texttt{Euler macro at n=N}\\
\texttt{Compute $X_{N+2},\dots$}\\
For $n=N+1:\lfloor t_{max}/H \rfloor$\\
~~~~ \texttt{Micro-integration}\\
~~~~~~~~~~~$u_{n,0}=X_n$ \% \texttt{initial value}\\
~~~~~~~~~~~For $\nu=-\nu_{max}:\nu_{max}$, $v_{n,\nu}=u_{n-N,\nu}$, end \% \texttt{save for later history}\\
~~~~~~~~~~~For $\nu=0:\nu_{max}-1$, $u_{n,\nu+1}=u_{n,\nu}+hf(u_{n,\nu},v_{n,\nu},t_n+\nu h,\Omega\nu h;\Omega)$\\
~~~~~~~~~~~$u_{n,-\nu-1}=u_{n,-\nu}-hf(u_{n,-\nu},v_{n,-\nu},t_n-\nu h,-\Omega\nu h;\Omega)$, end \% \texttt{Euler}\\
~~~~~~~~~~~$F_n=(u_{n,\nu_{max}}-u_{n,-\nu_{max}})/(2T)$ \% \texttt{slope for macro-step}\\
~~~~ \texttt{Macro-integration}\\
~~~~~~~~~~~$t_{n+1}=t_n+H$, $X_{n+1}=X_n+(3H/2)F_n-(H/2)F_{n-1}$  \% \texttt{Adams-Bashforth2 macro}\\
end\\
\hline
\end{tabular}
\end{center}
\label{tab:algorithm}
\end{table}

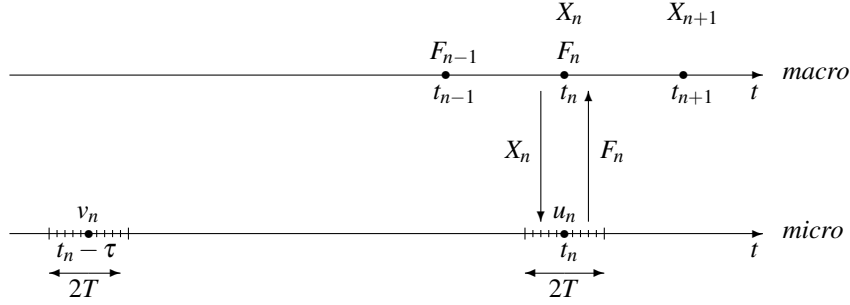
\begin{figure}
\begin{center}
\setlength{\unitlength}{1.5pt}
\begin{picture}(220,75)
\put(10,18){\vector(1,0){190}} \put(10,58){\vector(1,0){190}}
\multiput(120,58)(30,0){3}{\circle*{2}}
\put(140,16.5){\line(0,1){3}}
\put(150,18){\circle*{2}}
\put(160,16.5){\line(0,1){3}}
\multiput(142,17.1)(2,0){4}{\line(0,1){1.8}}
\multiput(152,17.1)(2,0){4}{\line(0,1){1.8}}
\put(149,12){$t_n$}\put(149,52){$t_n$}
\put(177,52){$t_{n+1}$}\put(117,52){$t_{n-1}$}
\put(150,8){\vector(-1,0){10}}\put(150,8){\vector(1,0){10}}
\put(146,2){$2T$}
\put(20,16.5){\line(0,1){3}}
\put(30,18){\circle*{2}}
\put(40,16.5){\line(0,1){3}}
\multiput(22,17.1)(2,0){4}{\line(0,1){1.8}}
\multiput(32,17.1)(2,0){4}{\line(0,1){1.8}}
\put(22,12){$t_{n}-\tau$}
\put(30,8){\vector(-1,0){10}}\put(30,8){\vector(1,0){8}}
\put(25,2){$2T$}
\put(148,62){$F_n$}\put(116,62){$F_{n-1}$}
\put(148,72){$X_n$}\put(176,72){$X_{n+1}$}
\put(147,22){$u_n$}\put(27,22){$v_n$}
\put(197,52){${t}$}\put(197,12){$t$}
\put(205,57){$macro$}\put(205,17){$micro$}
\put(144,54){\vector(0,-1){33}}
\put(156,21){\vector(0,1){33}}
\put(135,37){$X_n$}
\put(159,37){$F_n$}
\end{picture}
\end{center}
\caption{Schematic description of SAM. Once \(X_n\) is available, it is passed to the micro-integrator to find \(u_{n,\nu}\) for varying \(\nu\). Numerical differentiation of the micro-solution yields \(F_n\) which is used in the macro-stepping to compute \(X_{n+1}\).}
\label{figschematic}
\end{figure}

\section{The averaged system}
\label{sec:averaged}
In this section and the three that follow, we assume that the system \eqref{gensys} is of the particular form
\begin{equation}\label{specficsys}
\dot{x}(t)=f(x(t),y(t),\Omega t).
\end{equation}
When comparing with the general format in \eqref{gensys} we note that now \(f(x,y,\theta)\) has three
arguments rather than five. The case where \(f\) includes a slow explicit dependence on \(t\), i.e.\ \(f =
f(x,y,t,\theta)\), may be trivially reduced to \eqref{specficsys} by adding a component \(x_{D+1}\) to the
state vector \(x\in\mathbb{R}^D\) and setting \(\dot{x}_{D+1} = 1\). The case where \(f\) depends on
\(\Omega\), i.e.\ \(f = f(x,y,\theta;\Omega)\), is taken up in the final section. We assume that
$f(x,y,\theta)$   and the initial function \(\varphi\) in \eqref{eq:varphi} are of class
 \(C^3\), and that the solution \(x\) exists in the interval \([0,t_{max}]\) where \(t_{max}\) is a constant (i.e.\ does not change with the parameter \(\Omega\)).

Using an approach similar to that in \cite{Chartier}, \cite{Murua}, \cite{Sanz-Serna}, we use a Fourier decomposition
\[
f(x,y,\theta)=\sum_{k\in \mathbb{Z}}\exp(ik\theta)f_k(x,y).
\]
The coefficients $f_k(x,y)$ satisfy $f_k\equiv f_{-k}^*$ because the problem is real.

It is easily seen that, under the preceding hypotheses, $x$ undergoes oscillations of frequency $\Omega$ and
amplitude $\mathcal{O}(1/\Omega)$ as \(\Omega \rightarrow \infty\). To reduce the amplitude of the
oscillations to $\mathcal{O}(1/\Omega^{2})$, we consider the  near identity change of variables
\begin{equation}\label{varchange}
\left\{
\begin{aligned}
x&=X+\frac{1}{\Omega}\sum_{k\neq 0}\frac{\exp(ik\Omega t)-1}{ik}f_k(X,Y),~~~t \geq 0,\\
x&=X,~~~~~~~~~~~~~~~~~~~~~~~~~~~~~~~~~-\tau\le t<0,
\end{aligned}
\right.
,
\end{equation}
where we note that $x$ and $X$ coincide at stroboscopic times i.e.\ the change is stroboscopic.
In what follows, we use the notation $z(t)=y(t-\tau),~t\ge \tau$, $Y(t) = X(t-\tau), t\ge 0$,
$Z(t)=Y(t-\tau),~t\ge \tau$.

The proof of the following result is a straightforward but very lengthy exercise on Taylor expansion.

\begin{lemma}{\em
The change of variables (\ref{varchange}) transforms the system (\ref{specficsys}) into
\begin{eqnarray}
\dot{X}&=&f_0-\frac{1}{\Omega}\sum_{k\neq 0}\frac{\exp(ik\Omega t)-1}{ik}\frac{\partial f_k}{\partial X}f_0
-\frac{1}{\Omega}\sum_{k\neq 0}\frac{\exp(ik\Omega t)-1}{ik}\frac{\partial f_k}{\partial Y}\dot{\varphi}(t-\tau)\nonumber
\\
& &\qquad+
\frac{1}{\Omega}\sum_{k\neq 0}\frac{\exp(ik\Omega t)-1}{ik}\frac{\partial f_0}{\partial X}f_k
+\frac{1}{\Omega}\sum_{k,\ell\neq 0}\exp(ik\Omega t)\frac{\exp(i\ell\Omega t)-1}{i\ell}\frac{\partial f_k}{\partial X}f_{\ell}
\nonumber
\\
&&\qquad
+\mathcal{O}(\frac{1}{\Omega^2}),
\label{eq:l1}
\end{eqnarray}
for $0\le t< \tau$, and into
\begin{eqnarray}
\dot{X}&=&f_0
-\frac{1}{\Omega}\sum_{k\neq 0}\frac{\exp(ik\Omega t)-1}{ik}\frac{\partial f_k}{\partial X}f_0
-\frac{1}{\Omega}\sum_{k\neq 0}\frac{\exp(ik\Omega t)-1}{ik}\frac{\partial f_k}{\partial Y}f_{0\tau}\nonumber\\
& &\qquad
+\frac{1}{\Omega}\sum_{k\neq 0}\frac{\exp(ik\Omega t)-1}{ik}\frac{\partial f_0}{\partial X}f_k
+\frac{1}{\Omega}\sum_{k\neq 0}\frac{\exp(ik\Omega (t-\tau))-1}{ik}\frac{\partial f_0}{\partial Y}f_{k\tau}\nonumber
\\
& &\qquad+
\frac{1}{\Omega}
\sum_{k,\ell\neq 0}\exp(ik\Omega t)\frac{\exp(i\ell\Omega t)-1}{i\ell}
\frac{\partial f_k}{\partial X}f_{\ell}
\nonumber\\
& &\qquad
+\frac{1}{\Omega}\sum_{k,\ell\neq 0}\exp(ik\Omega t)\frac{\exp(i\ell\Omega
(t-\tau))-1}{i\ell}\frac{\partial f_k}{\partial Y}f_{\ell\tau}
 +\mathcal{O}(\frac{1}{\Omega^2}), \label{eq:l2}
\end{eqnarray}
 for \(t \geq \tau\).
Here  \[ f_k=f_k(X,Y),\qquad f_{k\tau}=f_{k}(Y,Z), \qquad k\in \mathbb{Z}.
\]
}
\end{lemma}

Note that system in \eqref{eq:l1}--\eqref{eq:l2} is of the (expected) form
\begin{equation}\label{eq:simple}
\dot X = f_0+\mathcal{O}(1/\Omega),
\end{equation}
where
\(f_0\) is the average
\begin{equation}\label{eq:integral}
 f_0(X,Y) = \frac{1}{2\pi}\int_0^{2\pi} f(X,Y,\theta)\,d\theta.
\end{equation}
By suppressing the remainder in \eqref{eq:simple}, we obtain the averaged system \(\dot X = f_0(X,Y)\)
with \(\mathcal{O}(1/\Omega)\) error; this is not sufficiently accurate for the values of \(\Omega\) of interest and
 we take the averaging procedure to higher order.
To do so, we perform an additional stroboscopic change of variables chosen so as to annihilate the oscillatory
parts of the $\mathcal{O}(1/\Omega)$ terms in \eqref{eq:l1}--\eqref{eq:l2}. Specifically we take:
\begin{eqnarray*}
X&=&\tilde{X}+\frac{1}{\Omega^2}\sum_{k\neq 0}\frac{\exp(ik\Omega t)-1}{k^2}
\frac{\partial f_k}{\partial \tilde{X}}f_0
+\frac{1}{\Omega^2}\sum_{k\neq 0}\frac{\exp(ik\Omega t)-1}{k^2}
\frac{\partial f_k}{\partial \tilde{Y}}\dot{\varphi}(t-\tau)
\\
& &
-\frac{1}{\Omega^2}\sum_{k\neq 0}\frac{\exp(ik\Omega t)-1}{k^2}
\frac{\partial f_0}{\partial \tilde{X}}f_k
-\frac{1}{\Omega^2}\sum_{\substack{k\neq 0\\\ell\neq 0,-k}}\frac{\exp(i(k+\ell)\Omega t)-1}{\ell(k+\ell)}
\frac{\partial f_k}{\partial \tilde{X}}f_\ell
\\&&
+\frac{1}{\Omega^2}\sum_{k,\ell\neq 0}\frac{\exp(ik\Omega t)-1}{k\ell}
\frac{\partial f_k}{\partial \tilde{X}}f_\ell
\end{eqnarray*}
for $0\le t< \tau$ and
\begin{eqnarray*}
X&=&
\tilde{X}
+\frac{1}{\Omega^2}\sum_{k\neq 0}\frac{\exp(ik\Omega t)-1}{k^2}\frac{\partial f_k}{\partial \tilde{X}}f_0
+\frac{1}{\Omega^2}\sum_{k\neq 0}\frac{\exp(ik\Omega t)-1}{k^2}\frac{\partial f_k}{\partial \tilde{Y}}f_{0\tau}
\\
&&-\frac{1}{\Omega^2}\sum_{k\neq 0}\frac{\exp(ik\Omega t)-1}{k^2}\frac{\partial f_0}{\partial \tilde{X}}f_k
-\frac{1}{\Omega^2}\sum_{k\neq 0}\frac{\exp(ik\Omega(t-\tau))-\exp(-ik\Omega\tau)}{k^2}
\frac{\partial f_0}{\partial \tilde{Y}}f_{k\tau}
\\
&&-\frac{1}{\Omega^2}\sum_{\substack{k\neq 0\\\ell\neq 0,-k}}\frac{\exp(i(k+\ell)\Omega t)-1}{l(k+\ell)}
\frac{\partial f_k}{\partial \tilde{X}}f_\ell
+\frac{1}{\Omega^2}\sum_{k,\ell\neq 0}\frac{\exp(ik\Omega t)-1}{k\ell}\frac{\partial f_k}{\partial \tilde{X}}f_\ell
\\
&&
-\frac{1}{\Omega^2}\sum_{\substack{k\neq 0\\\ell\neq 0,-k}}\exp(-il\Omega \tau)\frac{\exp(i(k+\ell)\Omega t)-1}{\ell(k+\ell)}
\frac{\partial f_k}{\partial \tilde{Y}}f_{\ell\tau}
+\frac{1}{\Omega^2}\sum_{k,\ell\neq 0}\frac{\exp(ik\Omega t)}{k\ell}\frac{\partial f_k}{\partial \tilde{Y}}f_{\ell\tau}
\end{eqnarray*}
for $t\ge \tau$, where $f_k=f_k(\tilde{X},\tilde{Y})$ and $f_{k\tau}=f_k(\tilde{X},\tilde{Y})$ ($\tilde{Y}(t)
=\tilde{X}(t-\tau)$, $\tilde{Z}(t) =\tilde{Y}(t-\tau)$).

Taking the last displays to \eqref{eq:l1}--\eqref{eq:l2} and discarding the $\mathcal{O}(1/\Omega^2)$
remainder results in the averaged system \eqref{aversys0}--\eqref{aversys1} below, where, for notational
simplicity, we have suppressed the tildes over $X$, $Y$ and $Z$ and   \([\cdot,\cdot]\) is the Lie-Jacobi
bracket or commutator defined by
$$[f_i,f_j]=\frac{\partial f_j}{\partial X}f_i-\frac{\partial f_i}{\partial X}f_j,~~~i,j\in \mathbb{Z}.
$$
The averaged solution \(X\) will obviously approximate \(x\) at stroboscopic times with
$\mathcal{O}(1/\Omega^2)$ errors (arising from discarding the $\mathcal{O}(1/\Omega^2)$ remainder);  this
implies that, at least for \(\Omega\) sufficiently large, \(X\) exists in the interval \([0,t_{max}]\). We
then have proved the following result.

\begin{theorem}{\em
\label{theorem1}
 Consider the averaged problem
\begin{eqnarray}\label{aversys0}
\dot{X}(t)&=&F(X(t),Y(t),Z(t),t;\Omega),\qquad t\ge 0,\\
Y(t) &=& X(t-\tau),\qquad t\geq 0,\nonumber\\
Z(t) &=& Y(t-\tau),\qquad t\geq \tau,\nonumber\\
X(t) &=& \varphi(t), \qquad -\tau\leq t\leq 0,\nonumber
\end{eqnarray}
 with \( F(X,Y,Z,t;\Omega)= F^{(1)}(X,Y,t;\Omega)\),
\begin{eqnarray}\label{aversys}
F^{(1)}(X,Y,t;\Omega)=f_0+\sum_{k>0}\frac{i}{k\Omega}\left([f_k-f_{-k},f_0]+[f_{-k},f_{k}]\right)-\sum_{k\neq0}\frac{i}{k\Omega}\frac{\partial f_k}{\partial Y}\dot{\varphi}(t-\tau),
\end{eqnarray}
for $0\le t<\tau$ and \( F(X,Y,Z,t;\Omega)= F^{(2)}(X,Y,Z;\Omega)\),
\begin{eqnarray}\label{aversys1}
F^{(2)}(X,Y,Z;\Omega)&=&f_0+\sum_{k>0}\frac{i}{k\Omega}\left([f_k-f_{-k},f_0]+[f_{-k},f_{k}]\right)-\sum_{k\neq0}\frac{i}{k\Omega}\frac{\partial f_k}{\partial Y}f_{0\tau}\\ \nonumber
& &\qquad+\sum_{k\neq0}\frac{i}{k\Omega}\frac{\partial f_0}{\partial Y}f_{k\tau}+\sum_{k\neq0}\frac{i\exp(ik\Omega \tau)}{k\Omega}\frac{\partial f_k}{\partial Y}f_{-k\tau},
\end{eqnarray}
for $t\ge \tau$.
For \(\Omega\) sufficiently large, the averaged solution \(X\) exists in the interval \([0,t_{max}]\). Furthermore, for the approximation at stroboscopic times, we may write
\[
\max_{ 0\leq t^{(k)}\leq t_{max}}\|x(t^{(k)})-X(t^{(k)}) \|=\mathcal{O}(\frac{1}{\Omega^2}),\qquad \Omega\rightarrow \infty.
\]
}
\end{theorem}

We point out that, for \(t\geq \tau\), the dependence of \(\dot X(t)\) on the past values \(X(s)\), \(-\tau\leq s< t\) is through {\em both} \(Y(t) = X(t-\tau)\) {\em and} \(Z(t) = X(t-2\tau)\); this is to be compared with the situation for the original oscillatory system \eqref{gensys}, which does not include the delay \(2\tau\).
The functions \(F^{(1)}\) and \(F^{(2)}\) are of class \(C^2\), but of course \(F\) is discontinuous at
\(t =\tau\). By implication, \(X(t)\) will be of class \(C^3\) except at \(t = 0,\tau\) (where \(\dot{X}(t)\) has a jump discontinuity), at \( t = 2\tau\) (where the second derivative jumps) and \(t = 3\tau\) (where the third derivative jumps).

In the particular case where \(f\) does not depend on the delayed argument \(y\), so that we are dealing with an ordinary differential system,
\eqref{aversys} and \eqref{aversys1} are in agreement with formula (64) of \cite{Murua}.

\section{Error analysis: exact micro-integration}
\label{sec:error1}
In this section we investigate  the global error of the algorithm under the assumption that the micro-integration
is exact, so that the macro-integration and the numerical differentiations performed to find the slopes \(F_n\)
are the only sources
of error. This scenario is of course relevant when the micro-step \(h\) is chosen to be very small.
The errors due to the
Euler micro-integration will be studied in the next section.

In order to avoid misunderstandings, we state
that \lq exact micro-integration\rq\ has to be
understood as follows. Consider e.g.\ the
computation of \(X_2\), \dots, \(X_N\) in
Table~\ref{tab:algorithm}; in this section we
assume that
\[
F_n = \frac{u_n(T)-u_n(-T)}{2T},
\]
where \(u_n\) solves the problem \(\dot u_n = f(u_n,v_n,\Omega t)\), \(u_n(0)= X_n\), \( v_n(t) =\varphi(-\tau+nH+t)\).
Of course similar modifications of the algorithm in Table~\ref{tab:algorithm} have to be carried out for the computation of \(F_n\) for other values of \(n\).

We begin by proving a stability bound for the macro-integrator.
We consider a  sequence \(\{\hat{X}_n\}\) of vectors in \(\mathbb{R}^D\) such that
\(\hat{X}_{-n}=\varphi(-nH)\), \(n=1,\dots,N\), and furthermore satisfy the macro-integration equations
with residuals \(\{ \hat{\rho}_n\}\), i.e.
\begin{eqnarray*}
\hat{X}_1&=&\hat{X}_0+HF^{(1)}(\hat{X}_0,\hat{X}_{-N},0;\Omega)+H\hat{\rho}_0,\\
\hat{X}_{n+1}&=&\hat{X}_n+\frac{3H}{2}F^{(1)}(\hat{X}_n,\hat{X}_{n-N},nH;\Omega)\\
&&\qquad\qquad-\frac{H}{2}F^{(1)}(\hat{X}_{n-1},\hat{X}_{n-1-N},(n-1)H;\Omega)+H\hat{\rho}_n,
\qquad n=1,\dots,N-1,\\
\hat{X}_{N+1}&=&\hat{X}_N+HF^{(2)}(\hat{X}_N,\hat{X}_0,\hat{X}_{-N};\Omega)+H\hat{\varrho}_N,\\
\hat{X}_{n+1}&=&\hat{X}_n+\frac{3H}{2}F^{(2)}(\hat{X}_n,\hat{X}_{n-N},\hat{X}_{n-2N};\Omega)\\
&&\qquad\qquad-\frac{H}{2}F^{(2)}(\hat{X}_{n-1},\hat{X}_{n-1-N},\hat{X}_{n-1-2N};\Omega)+H\hat{\rho}_n,
\qquad n\ge N+1.
\end{eqnarray*}
Furthermore we consider a second sequence \(\{\tilde{X}_n\}\) with \(\tilde{X}_{-n}=\varphi(-nH)\),
\(n=1,\dots,N\), satisfying the macro-integration equations above with residuals \(\{ \tilde{\rho}_n\}\), rather than \(\{ \hat{\rho}_n\}\).

\begin{proposition}{\em\label{prop:stability}
 To each bounded set \(B\subset\mathbb{R}^D\), there corresponds a constant \(C>0\),
 independent of \(H\) and \(\Omega\), such that for any  sequences \(\{\hat{X}_n\}\) and \(\{\tilde{X}_n\}\)
 as above contained in \(B\),
\[
\|\hat{X}_n-\tilde{X}_n \|\le C\,\sum_{k=0}^{n-1} H \|\hat{\rho}_k-\tilde{\rho}_k \|,\qquad 0\leq nH\leq t_{max}.
\]
}
\end{proposition}

\begin{proof} From the hypotheses in the  preceding section,
\(F\) is a Lipschitz
continuous function of \(X\), \(Y\) and \(Z\) with a Lipschitz constant that is uniform as \(t\) varies in the interval \(0\leq t\leq t_{max}\) and \(X\), \(Y\) and \(Z\) vary in \(B\).
The stability bound is then proved  in a standard way by recurrence.
\end{proof}

In our next result we investigate the consistency of the formulas \eqref{eq:slope1}--\eqref{eq:slope2} used to recover the slopes \(F_n\). There are four cases corresponding to the four successive blocks in Table~\ref{tab:algorithm}.

\begin{proposition}\label{prop:fourcases}{\em The following results hold:
\begin{enumerate}
\item If the function \(u_n\) solves the problem
\begin{eqnarray*}
\dot{u}_n(t)&=&f(u_n(t),\varphi(-\tau+t_n+t),\Omega t),\\
u_n(0)&=&X_n,
\end{eqnarray*}
then, with \(Y_n = \varphi(-\tau+t_n)\),
\begin{equation}\label{slope0}
\frac{u_n(T)-u_n(0)}{T}=f_0(X_n,Y_n)+\mathcal{O}(\frac{1}{\Omega}).
\end{equation}
\item  If  \(u_n\) is as in the preceding item,
then
\begin{equation}\label{slope1}
\frac{u_n(T)-u_n(-T)}{2T}=f_0+\sum_{k>0}\frac{i}{k\Omega}\left([f_k-f_{-k},f_0]+[f_{-k},f_{k}]\right)
-\sum_{k\neq0}\frac{i}{k\Omega}\frac{\partial f_k}{\partial Y}\dot{\varphi}(-\tau+t_n)
+\mathcal{O}(\frac{1}{\Omega^2}),
\end{equation}
where \(f_k\), \(k\in\mathbb{Z}\) are evaluated at \((X_n,Y_n)\), \(Y_n = \varphi(-\tau+t_n)\).
\item If  \(u_n\) and \(v_n\) satisfy
 \begin{eqnarray*}
\dot{u}_n(t)&=&f(u_n(t),v_n(t),\Omega t),\\
\dot{v}_n(t)&=&f(v_n(t),w_n(t),\Omega t),\\
u_n(0)&=&X_n,\\
v_n(0)&=&Y_n,\\
w_n(0)&=&Z_n,
\end{eqnarray*}
with \(w_n\) a continuously differentiable function, then
\begin{equation}\label{slope2}
\frac{u_n(T)-u_n(0)}{T}=f_0(X_n,Y_n)+\mathcal{O}(\frac{1}{\Omega}).
\end{equation}
\item If  \(u_n\) and \(v_n\) are as in 3., then
\begin{eqnarray}\label{slope3}\nonumber
\frac{u_n(T)-u_n(-T)}{2T}&=&f_0+\sum_{k>0}\frac{i}{k\Omega}\left([f_k-f_{-k},f_0]+[f_{-k},f_{k}]\right)-\sum_{k\neq0}\frac{i}{k\Omega}\frac{\partial f_k}{\partial Y}f_{0\tau}\\
& &\qquad+\sum_{k\neq0}\frac{i}{k\Omega}\frac{\partial f_0}{\partial Y}f_{k\tau}+\sum_{k\neq0}\frac{i}{k\Omega}\frac{\partial f_k}{\partial Y}f_{-k\tau}+\mathcal{O}(\frac{1}{\Omega^2}),
\end{eqnarray}
where \(f_k\), \(k\in\mathbb{Z}\) are evaluated at \((X_n,Y_n)\) and \(f_{k\tau}\) stands for \(f_k(Y_n,Z_n)\).
\end{enumerate}
}
\end{proposition}

\begin{proof}
Here we prove the fourth case; the other proofs are similar (and slightly simpler). For simplicity the
subindex \(n\) is dropped. We rewrite the equation for \(u\)
 in integral form
\[
u(t)=X+\int_0^t f(u(s),v(s),\Omega s)ds = \int_0^t \sum_k \exp(ik\Omega s) f_h(u(s),v(s))ds.
\]
and use Picard's iteration.  Clearly for \(-T\leq t \leq T\),
$$
u(t)=X+\mathcal{O}(\frac{1}{\Omega}).
$$
We take this equality to the integral equation above and integrate with respect to $s$  to find
$$
u(t)=X+tf_0(X,Y)+\sum_{k\neq0}\alpha_k(t)f_k(X,Y)+\mathcal{O}(\frac{1}{\Omega^2}),
$$
where
$$
\alpha_k(t)=\frac{\exp(ik\Omega t)-1}{ik\Omega},~~~k\in \mathbb{Z}.
$$
Then,
\begin{eqnarray}\nonumber
u(t)&=&X+\int_0^t \sum_{k\in\mathbb{Z}}\exp(ik\Omega s)\\ &&\nonumber\qquad \times
f_k\big(X+sf_0+\sum_{m\neq0}\alpha_m(s)f_m+\mathcal{O}(\frac{1}{\Omega^2}),
Y+sf_{0\tau}
+\sum_{m\neq0}\alpha_m(s)f_{m\tau}+\mathcal{O}(\frac{1}{\Omega^2})\big)ds.
\end{eqnarray}
By Taylor expanding $f$ at $X,Y$, we obtain,  for \(-T\leq t \leq T\):
\begin{eqnarray*}
u(t)&=&X+tf_0+\sum_{k\neq 0}\alpha_k(t)f_k+\int_0^t s \frac{\partial f_0}{\partial X}f_0ds
+\int_0^t \sum_{k\neq0}\alpha_{k}(s)\frac{\partial f_0}{\partial X}f_kds
\\
& &\qquad+\int_0^t \sum_{k\neq0}\alpha_{k}(s)\frac{\partial f_0}{\partial Y}f_{k\tau}ds
+\int_0^t s\frac{\partial f_0}{\partial Y}f_{0\tau}ds
+\int_0^t \sum_{k\neq0}\exp(ik\Omega s)s\frac{\partial f_k}{\partial X}f_{0}ds
\\
& &\qquad+\int_0^t
\sum_{k,m\neq0}\exp(ik\Omega s)\alpha_{m}(s)\frac{\partial f_k}{\partial X}f_{m}ds+\int_0^t\sum_{k\neq0}\exp(ik\Omega s)s\frac{\partial f_k}{\partial Y}f_{0\tau}ds
\\
&&\qquad
+\int_0^t \sum_{k,m\neq0}\exp(ik\Omega s)\alpha_m(s)\frac{\partial f_k}{\partial Y}f_{m\tau}ds
+\mathcal{O}(\frac{1}{\Omega^3}).
\end{eqnarray*}
The proof concludes by evaluating this expression at \(t =\pm T\) and taking those values to the left-hand
side of \eqref{slope3}.
\end{proof}

According to this result, \eqref{slope0} and \eqref{slope2} provide
approximations with $\mathcal{O}(1/\Omega)$ errors to \eqref{aversys}
and \eqref{aversys1} (evaluated at \(X=X_n\), \(Y=Y_n\), \(t=t_n\)) respectively,
as expected from the use of forward differencing.
Similarly  \eqref{slope1} approximates \eqref{aversys} (at \(X=X_n\), \(Y=Y_n\), \(t=t_n\)) with  $\mathcal{O}(1/\Omega^2)$ error,
as expected of central differences.  However, when comparing \eqref{slope3} with \eqref{aversys1} (at \(X=X_n\), \(Y=Y_n\), \(Z=Z_n\), \(t=t_n\)),
we observe that the last sum in \eqref{aversys1} has a factor \(\exp(ik\Omega\tau)\), which is not present
in the last sum in
\eqref{slope3} and therefore the error is, in general, only $\mathcal{O}(1/\Omega)$. To
achieve $\mathcal{O}(1/\Omega^2)$ errors
we may assume that the functions $f_k(X,Y)$, $k\neq 0$, are independent
of the second argument $Y$, i.e.\ the delay argument \(y\) only appears in \(f\) through
\(f_0\) (this is the case in \eqref{toggle}). Alternatively, we may assume
that, for all \(k\neq 0\), \(\exp(ik\Omega\tau)=1\), i.e.\ that \(\tau\) is an integer multiple
of the period \(T = 2\pi/\Omega\).

We are now ready to give the main  result of this section.

\begin{theorem}\label{th:main}{\em
Assume that the problem \eqref{specficsys}, with the smoothness assumptions stated
in the preceding section, is integrated with SAM with exact micro-integrations. In addition assume that one of the following hypotheses
holds:
\begin{itemize}
\item (H1) The oscillatory Fourier coefficients \(f_k\), \(k\neq 0\) of \(f\) do not depend on the delayed argument
\(y\).
\item (H2) The delay  \(\tau\) is an integer multiple of the period \(T = 2\pi/\Omega\).
\end{itemize}

Then there exist constants \(\Omega_0\), \(H_0\) and \(K\) such that, for \(\Omega\geq \Omega_0\) , \(H\leq H_0\),
\(0\leq t_n=nH\leq t_{max}\), the difference between the numerical solution and the solution of the averaged problem
 has the bound
\begin{equation}\label{eq:thbound}
\|X_n-X(t_n) \|\le K\left( H^2+\frac{H}{\Omega}+\frac{1}{\Omega^2}\right).
\end{equation}
}
\end{theorem}
\begin{proof}We apply Proposition~\ref{prop:stability} with \(B\) taken as a large ball containing the trajectory
\(X(t)\), \(0\leq t\leq t_{max}\) in its interior; the vectors \(X(t_n)\) play the role of \(\tilde{X}_n\) and
the vectors \(X_n\) play the role of \(\hat{X}_n\). It is clear that each \(\tilde{X}_n\) is in \(B\). For
\(H\) sufficiently small and \(\Omega\) sufficiently large the same will be true for each \(\hat{X}_n\); this is
established by means of a standard argument by contradiction using \eqref{eq:thbound}.

Each residual \(\tilde{\rho}_n\) is
\(\mathcal{O}(H^2)\) with the exceptions of
\(\tilde{\rho}_0\),  \(\tilde{\rho}_N\) and
\(\tilde{\rho}_{2N}\); these are only
\(\mathcal{O}(H)\) because the first two
correspond to Euler steps and in the third there
is a jump discontinuity in the second time
derivative of the averaged solution. According to
Proposition~\ref{prop:fourcases}, the residuals
\(\hat{\rho}_n\) for the numerical solution are
\(\mathcal{O}(1/\Omega^2)\), with the exceptions
of \(\hat{\rho}_0\) and \(\hat{\rho}_N\), which
are of size \(\mathcal{O}(1/\Omega)\). Taking
these results to
Proposition~\ref{prop:stability}, we get a global
error bound of the  desired form.
\end{proof}

{\em Remark 1.} It is clear that the bound in
\eqref{eq:thbound} may be replaced by one of the
form \(K^\prime (H^2+1/\Omega^2)\). We prefer the
form \eqref{eq:thbound}, as it relates to  three
sources of error:
 the macro-integration \(H^2\) error, the error \(H/\Omega\) arising from  differencing at \(0\),
  \(\tau\), \(2\tau\), and the error from second-order differencing at all other step points \(t_n\).

{\em Remark 2.} The discrepancy between \eqref{slope3} and \eqref{aversys1} that leads to the introduction of
(H1) and (H2) stems from the fact that the values of the phase \(\theta\) at \(t_n\) and \(t_{n-N} =
t_n-\tau\) are in general different in the oscillatory problem but the same in SAM. The hypothesis (H1) holds
in most applications;
in fact in all the papers mentioned in the introduction, the delay system is obtained by adding to
the right-hand side of an oscillatory ordinary differential system a linear feedback term $My$, with $M$ a
constant matrix. (H2) is relevant in those studies where there is freedom in choosing
the exact value of the large frequency \(\Omega\).

 {\em
Remark 3.} If (H1) and (H2) do not hold, the same proof yields for our algorithm a bound of the form
\(K(H^2+1/\Omega)\) under the assumption that \(f\) and \(\varphi\) are \(C^2\) functions. Numerical
experiments {reported in Section~\ref{sec:numer}} reveal that in that case the bound cannot be
improved to \(K(H^2+1/\Omega^2)\). However
 if (H1) and (H2) do not hold, errors of size
\(K(H^2+1/\Omega)\) may be obtained by means of a simpler algorithm
 based on applying forward
differences at each step point \(t_n\); obviously that alternative algorithm does not require the backward
integration legs \eqref{backward}.

\section{Error analysis: micro-integration errors}
\label{sec:error2}
We now take into account the errors introduced by the Euler micro-integration. We begin with an auxiliary result.
Note the improved error bound at the end of the integration interval.

\begin{proposition}{\em Consider the application of Euler's rule with constant step size \(h = T/\nu_{max}\) to integrate in the interval
\(0\leq t\leq T\) the initial value problem \(\dot u = f(u,v,\Omega t)\), \(u(0) = X\), where \(v\) is a known \(C^1\) function. Denote by \(u_\nu\) the Euler solution at \(t =\nu h\). There are constants \(C\), \(\Omega_0\) and \(h_0\) such that for \(h\leq h_0\), \(\Omega\geq \Omega_0\), the following bounds hold:
\begin{eqnarray}
 \| u_{\nu}-u(\nu h)\| &\leq& C h,\qquad \nu = 0, 1,\dots, \nu_{max},\label{eq:auxeuler}\\
 \| u_{\nu_{max}}-u(T)\| &\leq& C \frac{h}{\Omega},\label{eq:auxeulerbis}\\
\left \| \frac{u_{\nu_{max}}-X}{T}-\frac{u(T)-X}{T}\right\| &\leq& C h.\label{eq:auxeulerter}
\end{eqnarray}
}
\end{proposition}

\begin{proof} A standard error bound for Euler's rule is
\[
\| u_{\nu}-u(\nu h)\| \leq \frac{\exp(LT)-1}{L} M h,\qquad \nu = 0,\dots,\nu_{max},
\]
where \(L\) is the Lispchitz constant of \(f\) with respect to \(u\) in a neighbourhood
of the solution and \(M\) is an upper bound for \(\|(1/2) \ddot u(t)\|\), \(0\leq t\leq T\). In the present
 circumstances we have to take into account that, as \(\Omega\rightarrow \infty\), the length \(T=2\pi/ \Omega\) of
  the integration interval decreases and \(M\) grows like \(\Omega\), because
\[\ddot u = \frac{\partial f}{\partial u} \dot u+\frac{\partial f}{\partial v} \dot v+\Omega \frac{\partial f}{\partial t}.\]

From the elementary inequality \((\exp(LT)-1)/L \leq T\exp(LT)\) and the standard bound, we have
\[
\| u_{v}-u(\nu h)\| \leq T \exp(LT)M h,\qquad \nu = 0,\dots,\nu_{max},
\]
and therefore \eqref{eq:auxeuler} holds.

By adding all the Euler equations we find
\[
u_{\nu_{max}} = X + \sum_{\nu=0}^{\nu_{max}-1} h f(u_\nu,v(\nu h),\Omega \nu h),
\]
and from \eqref{eq:auxeuler},
\begin{equation}\label{eq:auxeuler1}
u_{v_{max}} = X +\sum_{\nu=0}^{\nu_{max}-1} h f(u(\nu h),v(\nu h),\Omega \nu h) +\mathcal{O}(hT),
\end{equation}
a relation that has to be compared with
\begin{equation}\label{eq:auxeuler2}
u(T) = X +\int_0^T f(u(s),v(s), \Omega s)\, ds.
\end{equation}

The bound \eqref{eq:auxeulerbis} will be established if we show that the quadrature sum in \eqref{eq:auxeuler1} approximates the integral in
\eqref{eq:auxeuler2} with errors of size \(\mathcal{O}(hT)\). To this end we decompose the function being integrated as
\[
f(u(s),v(s),\Omega s) = f(X,v(0),\Omega s) + \Big( f(u(s),v(s), \Omega s) -f(X,v(0),\Omega s)\Big) = f_1+f_2.
\]
It is easily seen (for instance by expanding in a Fourier series) that the total time derivative  \((d/dt)f_2\) remains bounded as \(\Omega\rightarrow \infty\); elementary
results then show that the quadrature of \(f_2\) has errors of the desired size \(\mathcal{O}(hT)\). On the other hand,
the time derivative of \(f_1\) grows like \(\Omega\), and quadrature errors of size \(\mathcal{O}(h)\) may be feared. Fortunately  the quadrature for \(f_1\) is actually exact, because one  checks by an explicit computation that it is exact for each Fourier mode
\(f_k(X,v(0)) \exp(ik\Omega s)\).

The third bound \eqref{eq:auxeulerter} is a trivial consequence of \eqref{eq:auxeulerbis}.
\end{proof}

It goes without saying that the corresponding
result holds for backward integrations with
\(-T\leq t\leq 0\). The following theorem
provides the main result of this paper.

\begin{theorem}\label{th:main2}
{\em
Assume that the problem \eqref{specficsys}, with the smoothness assumptions stated
in the preceding sections, is integrated with SAM. In addition assume that one of the following hypotheses
holds:
\begin{itemize}
\item (H1) The oscillatory Fourier coefficients \(f_k\), \(k\neq 0\) of \(f\) do not depend on the delayed argument
\(y\).
\item (H2) The delay  \(\tau\) is an integer multiple of the period \(T = 2\pi/\Omega\).
\end{itemize}

Then there exist constants \(\Omega_0\), \(H_0\), \(h_0\) and \(K\) such that, for \(\Omega\geq \Omega_0\) , \(H\leq H_0\), \(h\leq h_0\),
\(0\leq t_n=nH\leq t_{max}\), the difference between the numerical solution and the solution of the averaged problem
 may be bounded as follows
 \[
 \|X_n-X(t_n) \| \le K\left( H^2+\frac{H}{\Omega}+\frac{1}{\Omega^2}+h\right).
 \]
 In particular, if the grids are refined in such a way that \(h\) is taken proportional to \(H/\Omega\) (i.e.\ \(\nu_{max}\) is taken proportional to \(N\)), then the bound becomes
\begin{equation}\label{eq:b}
\|X_n-X(t_n) \| \le K^\prime\left( H^2+\frac{H}{\Omega}+\frac{1}{\Omega^2}\right).
\end{equation}
}
\end{theorem}
\begin{proof}We argue as in Theorem~\ref{th:main}. Now in the residual for the numerical solution \(\{X_n\}\) we have to taken into account the micro-integration error. For \(n\leq N\), the bound \eqref{eq:auxeulerter} (and the corresponding bound for the backward integration) show that the micro-integration adds a term of size \(\mathcal{O}(h)\) to the residual. For \(n>N\) the situation is slightly more complicated, because the algorithm uses past values \(v_{n,\nu}\) that are themselves affected by micro-integration errors. However the stability of the micro-integrator guarantees
 that even when those errors are taken into account an estimate like \eqref{eq:auxeulerter} holds.
\end{proof}

We recall that taking \(h\)  proportional to
\(H/\Omega\) makes the complexity of the
algorithm independent of \(\Omega\).

\section{Numerical experiments}
\label{sec:numer}
\begin{table}[p]
\caption{Errors in $x_1$ for SAM with respect to the averaged solution  for problem \eqref{toggle}. Step points are not stroboscopic times}
\footnotesize
\vspace{-5mm}
\begin{center}
\resizebox{\textwidth}{!}{
\begin{tabular}{rcccccccccc}
\hline
N   &$\Omega=25$ &$\Omega=50$ & $\Omega=100$ & $\Omega=200$ & $\Omega=400$ & $\Omega=800$ & $\Omega=1600$ & $\Omega=3200$
\\ \hline 1&6.28(-2)&3.42(-2)&1.71(-2)& 7.87(-3) & 3.14(-3) & 1.66(-3) & 2.04(-3) &2.30(-3)\\ 2&***&7.66(-3)&3.74(-3)&1.66(-3)& 8.27(-4) & 7.56(-4) & 7.20(-4)  & 7.02(-4)\\ 4&***&***&1.11(-3)& 4.45(-4) &2.60(-4)& 2.20(-4) & 1.99(-4) &1.88(-4)\\ 8&***&***&***&1.80(-4)& 6.35(-5) & 5.57(-5) & 5.06(-5) & 4.77(-5)\\ 16&***&***&***&***& 3.20(-5) & 1.27(-5) & 1.22(-5) & 1.18(-5)\\32&***&***&***&***& *** & 6.31(-6) & 2.81(-6) &2.85(-6)\\ 64&***&***&***&***& *** & *** & 1.36(-6) &6.46(-7)\\ 128&***&***&***&***& *** & *** & *** &3.22(-7)\\
\hline
\end{tabular}}
\end{center}
\label{tab4}
\vspace{1cm}
\end{table}

\begin{figure}[p]
\vspace{-4cm}
\centering\includegraphics[scale=0.45]{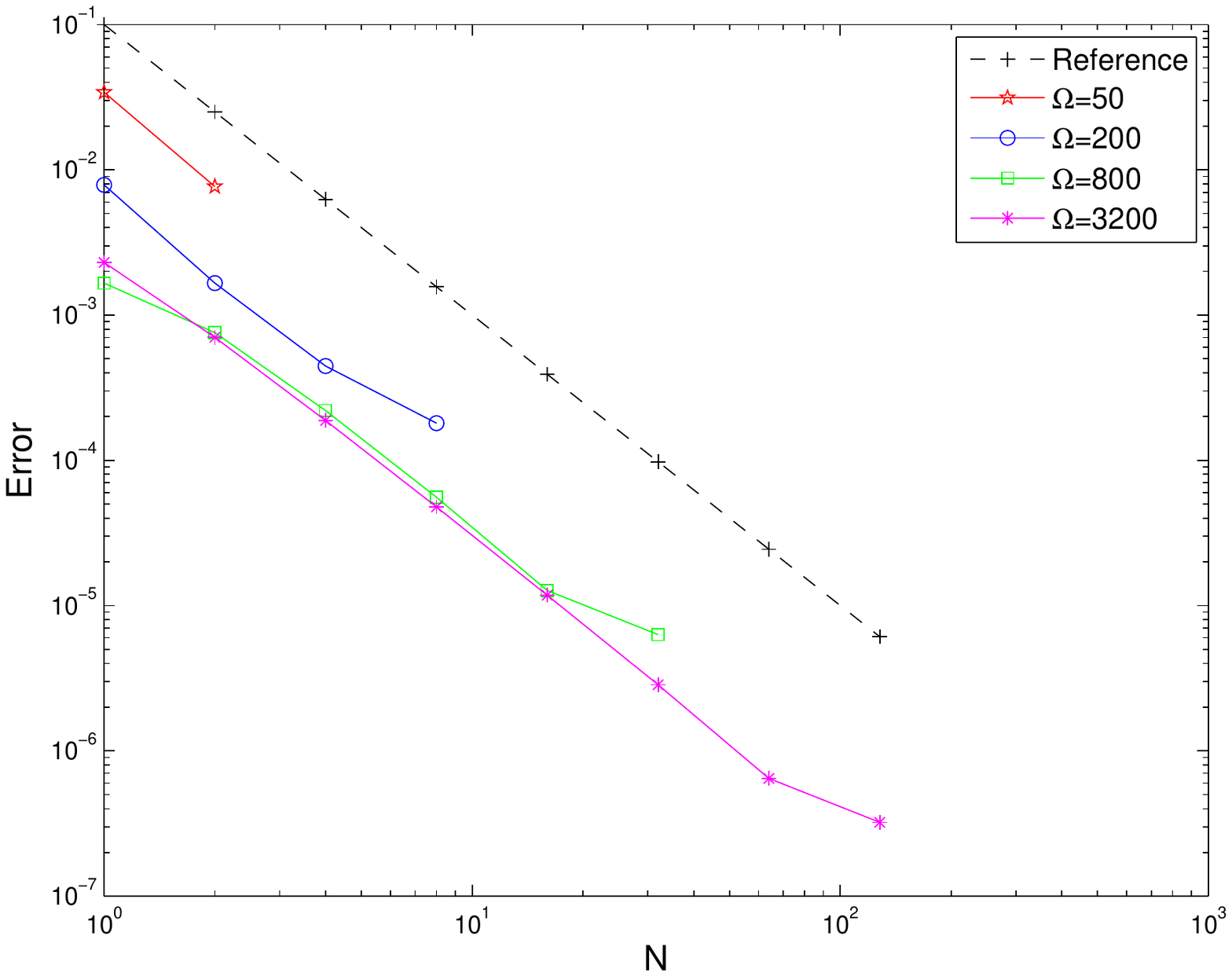}
\vspace{-3cm}
\caption{Error in the first component of the SAM solution with respect to the solution of the averaged
problem versus macro stepsize $H=\tau/N$ for some of the simulations in Table~\ref{tab4}.
Different lines correspond to difference values of $\Omega $ . Here  the reference line shows the \(N^{-2}\) behaviour}
\label{fig:err}
\end{figure}
 We now present experiments that illustrate the preceding theory.

 We first integrated with SAM the toggle switch problem \eqref{toggle} with $\alpha=2.5$, $\beta=2$,
$A=0.1$, $\omega=0.1$, $B=4.0$, and $\tau=0.5$. The function \(\varphi\) is constant: \( x_1(t) = 0.5\),
\(x_2(t) = 2.0\), \(-\tau \leq t\leq 0\), which corresponds to the system staying at an equilibrium up to time
\(t=0\) and then switching on the slow and fast oscillatory forcing terms. For the
 macro-stepsize we set \(H = \tau/N\), \(N=1,2,4,\dots\)
and the micro-stepsize was chosen as $h=T/(2N)$ (note that for the coarsest macro-step size \(N=1\), there are
only two Euler steps per period).
 Simulations took place in the interval $0\leq t\leq 2$ that includes the locations $t = 0$,
$t=\tau$ and $t=2\tau$ where the growth in the error of the SAM solution is more pronounced due to the lack of
smoothness. SAM has no difficulty in coping with the longer intervals required in practical simulations, but
we have chosen a short interval because on longer intervals it may be very expensive to obtain a sufficiently
accurate reference solution to measure errors; see in this connection the computational times quoted at the
end of Section~\ref{sec:extensions}.

Table~\ref{tab4} reports, for varying \(\Omega\) and \(N\), the maximum error, over $0\le t\le 2$,
in the \(X_1\) component of the SAM solution
with respect to the averaged solution obtained by
integrating \eqref{vraver} (this integration was carried out with the Matlab function dde23
with relative and absolute tolerances
$10^{-8}$ and $10^{-10}$ respectively).  The combinations of \(N\) and \(\Omega\) leading to values of \(H\) not significantly larger
than \(T\) were not attempted, as the HMM idea does not make sense for them.
Note that here \(\tau/T\) is irrational and therefore the step points \(t_n\) are not
stroboscopic times. Fig.~\ref{fig:err}
displays the errors in Table~\ref{tab4} as functions of \(N\); for clarity
not all values of \(\Omega\) are included.
By looking at the columns of the table (or at each of the four solid lines in the figure),
 we see that the error behaves as \(N^{-2}\), i.e\ as \(H^2\), except at the
bottom of each column, where the behaviour is as \(N^{-1}\).
This is of course the behaviour of the bound in \eqref{eq:b}.
Errors along rows saturate if \(\Omega\) is very large;
for those values one just observes the error in the macro-integration. This behaviour is also seen in the figure
by comparing points corresponding to the same value of \(N\) and varying \(\Omega\).
Along the main diagonal of the table,
errors  approximately divide by four, which is also in agreement with the bound \eqref{eq:b}. In
the figure this corresponds to observing the behaviour of the right-most point of each of the solid lines.

Table~\ref{tab2} differs from Table~\ref{tab4} in that now  \(\Omega\) is taken from the
sequence \(8\pi\), \(16\pi\), \dots\ that consists of values not very different from those in Table~\ref{tab4}.
In fact the errors in Table~\ref{tab2} are
very similar to those in Table~\ref{tab4}. However for the sequence \(8\pi\), \(16\pi\), \dots\  the step points are
stroboscopic times and it makes sense to compare the SAM solution with the true oscillatory solution. The
results are reported in Table~\ref{tab3}.
From Theorems~\ref{theorem1} and \ref{th:main2} the errors with respect to the true solution possess a
bound of the form \eqref{eq:b} and this is consistent with the data in the table.

\begin{table}[p]
\caption{Errors in $x_1$ for SAM with respect to the averaged solution  for problem \eqref{toggle}. Step points are stroboscopic times}
\footnotesize
\vspace{-5mm}
\begin{center}
\resizebox{\textwidth}{!}{
\begin{tabular}{rcccccccccc}
\hline
N   &$\Omega=8\pi$ &$\Omega=16\pi$ & $\Omega=32\pi$ & $\Omega=64\pi$ & $\Omega=128\pi$ & $\Omega=256\pi$ & $\Omega=512\pi$ & $\Omega=1024\pi$
\\ \hline 1&6.25(-2)&3.40(-2)&1.70(-2)& 7.82(-3) & 3.11(-3) & 1.66(-3) & 2.04(-3) &2.30(-3)\\ 2&***&7.62(-3)&3.72(-3)&1.65(-3)& 8.26(-4) & 7.56(-4) & 7.20(-4)  & 7.02(-4)\\ 4&***&***&1.11(-3)& 4.42(-4) &2.59(-4)& 2.20(-4) & 1.99(-4) &1.88(-4)\\ 8&***&***&***&1.78(-4)& 6.34(-5) & 5.57(-5) & 5.06(-5) & 4.77(-5)\\ 16&***&***&***&***& 3.16(-5) & 1.26(-5) & 1.22(-5) & 1.18(-5)\\32&***&***&***&***& *** & 6.24(-6) & 2.80(-6) &2.85(-6)\\ 64&***&***&***&***& *** & *** & 1.35(-6) &6.47(-7)\\ 128&***&***&***&***& *** & *** & *** &3.18(-7)\\
\hline
\end{tabular}}
\end{center}
\label{tab2}
\end{table}
\begin{table}[p]
\caption{Errors in $x_1$ for SAM with respect to the true oscillatory solution for problem \eqref{toggle}. Step points are stroboscopic times}
\footnotesize
\vspace{-5mm}
\begin{center}
\resizebox{\textwidth}{!}{
\begin{tabular}{rcccccccccc}
\hline
N   &$\Omega=8\pi$ &$\Omega=16\pi$ & $\Omega=32\pi$ & $\Omega=64\pi$ & $\Omega=128\pi$ & $\Omega=256\pi$ & $\Omega=512\pi$ & $\Omega=1024\pi$
\\ \hline 1&6.10(-2)&3.30(-2)&1.66(-2)& 7.74(-3) & 3.09(-3) & 1.66(-3) & 2.04(-3) &2.30(-3)\\ 2&***&6.65(-3)&3.41(-3)&1.56(-3)& 8.31(-4) & 7.57(-4) & 7.20(-4)  & 7.02(-4)\\ 4&***&***&7.95(-4)& 3.56(-4) &2.63(-4)& 2.21(-4) & 1.99(-4) &1.88(-4)\\ 8&***&***&***&9.25(-5)& 6.62(-5) & 5.64(-5) & 5.07(-5) & 4.77(-5)\\ 16&***&***&***&***& 1.50(-5) & 1.34(-5) & 1.23(-5) & 1.18(-5)\\32&***&***&***&***& *** & 3.03(-6) & 2.95(-6) &2.88(-6)\\ 64&***&***&***&***& *** & *** & 6.44(-7) &6.76(-7)\\ 128&***&***&***&***& *** & *** & *** &1.43(-7)\\
\hline
\end{tabular}}
\end{center}
\label{tab3}
\end{table}

In order to check numerically that the hypotheses (H1)--(H2) are necessary to ensure a bound of the form
\eqref{eq:b}, we have considered the  simple scalar equation
\begin{equation}\label{newpro}
\dot{x}=y+(x-y)\sin(\Omega t)+\frac{y}{2}\cos(2\Omega t);
\end{equation}
this is an academic example where (H1) does not hold (recall that in all the systems from the literature cited
(H1) holds). The averaged version is
$$
\dot{X}=Y-\frac{1}{\Omega}Y,
$$
for $0\le t<\tau$ and
$$
\dot{X}=Y+\frac{1}{\Omega}\Big(\frac{Y}{2}-\frac{Z}{2}\Big)\sin(\Omega\tau)-\frac{1}{16\Omega}Z\sin(2\Omega \tau),
$$
for $t\ge \tau$.

We used $\tau = 0.5$, $H = \tau/N$, $h = T/(5N)$, a constant $\varphi = 0.1$ and, as before, measured errors
in the maximum norm for $0\leq t\leq 2$. Table~\ref{tab:h2}\footnote{For typographic reasons only, this table
and the next have one column less than those presented before. There is nothing unexpected in the results
corresponding to the omitted frequency \(\Omega = 1024\pi\) (or \(\Omega = 1024\pi+2\pi\)). } gives errors
when $\Omega$ is taken from the sequence $8\pi$, $16\pi$, \dots, so that the periods $T$ are $1/4$, $1/8$,
\dots Hypothesis (H1) does not hold, but (H2) does, so that Theorem~\ref{th:main2} may be applied. In fact an
$N^{-2}$ (or equivalently $\Omega^{-2}$) behaviour is  seen along the diagonals of the table.

\begin{table}
\caption{Errors between SAM solution and the averaged solution for problem \eqref{newpro}. (H2) holds}
\vspace{-5mm}
\begin{center}
\resizebox{\textwidth}{!}{
\begin{tabular}{lccccccccc}
\hline
N  &$\Omega=8\pi$ &$\Omega=16\pi$ & $\Omega=32\pi$ & $\Omega=64\pi$ & $\Omega=128\pi$
& $\Omega=256\pi$ & $\Omega=512\pi$
\\ 1&3.08(-2)&3.00(-2)& 2.97(-2) & 2.95(-2) & 2.95(-2) & 2.95(-2) &2.94(-2)
\\ 2&***&8.41(-3)&8.36(-3)&8.35(-3)& 8.34(-3) & 8.33(-3) & 8.33(-3)
\\ 4&***&***&2.21(-3)& 2.26(-3) & 2.28(-3) & 2.29(-3) &2.29(-3)
\\ 8&***&***&***& 5.64(-4) & 5.84(-4) & 5.91(-4) & 5.94(-4)
\\ 16&***&***&***& *** & 1.42(-4) & 1.48(-4) & 1.50(-4)
\\ 32&***&***&***& *** & ***& 3.57(-5) &3.73(-5)
\\ 64&***&***&***& *** & *** & *** &8.94(-6)\\
\hline
\end{tabular}}
\end{center}
\label{tab:h2}
\end{table}

We next slightly changed the frequencies and used the sequence $8\pi+\pi/64$, $16\pi+\pi/32$, \dots  (this
represents in increase of less than  $0.2\%$ in each frequency). Neither (H1) nor (H2) are fulfilled and
Theorem~\ref{th:main2} may not be applied. The results in Table~\ref{tab:noh2} show that, for $\Omega$ large
the second order behaviour along the main diagonal is lost, indicating that the error does not behave as in
the bound \eqref{eq:b}. Note the substantial difference between Tables~\ref{tab:h2} and \ref{tab:noh2} at
$N=64$, in spite of the very small relative change in the value of \(\Omega\).

\begin{table}
\caption{Errors in $x$ between SAM solution and the averaged solution for problem \eqref{newpro}.
 Neither (H1) nor (H2) hold}
\vspace{-5mm}
\begin{center}
\resizebox{\textwidth}{!}{
\begin{tabular}{lcccccccc}
\hline
N  &$\Omega=8\pi+\frac{\pi}{64}$ &$\Omega=16\pi+\frac{\pi}{32}$ & $\Omega=32\pi+\frac{\pi}{16}$ & $\Omega=64\pi+\frac{\pi}{8}$ & $\Omega=128\pi+\frac{\pi}{4}$ & $\Omega=256\pi+\frac{\pi}{2}$ & $\Omega=512\pi+\pi$
\\ 1&3.08(-2)&3.00(-2)& 2.97(-2) & 2.95(-2) & 2.95(-2) & 2.95(-2) &2.95(-2)
\\ 2&***&8.43(-3)&8.38(-3)&8.36(-3)& 8.36(-3) & 8.36(-3) & 8.36(-3)  \\ 4&***&***&2.23(-3)& 2.28(-3) & 2.30(-3) & 2.31(-3) &2.32(-3) \\ 8&***&***&***& 5.78(-4) & 6.00(-4) & 6.12(-4) & 6.23(-4)\\ 16&***&***&***& *** & 1.58(-4) & 1.69(-4) & 1.79(-4)\\32&***&***&***& *** & ***& 5.63(-5) &6.60(-5)\\ 64&***&***&***& *** & *** & *** &3.76(-5)\\
\hline
\end{tabular}}
\end{center}
\label{tab:noh2}
\end{table}

\section{Extensions}
\label{sec:extensions}

\begin{table}
\caption{Errors in $x_1$ for SAM with respect to the averaged solution for problem \eqref{geneproblem}.}
\footnotesize
\vspace{-5mm}
\begin{center}
\resizebox{\textwidth}{!}{
\begin{tabular}{rcccccccccc}
\hline
N   &$\Omega=8\pi$ &$\Omega=16\pi$ & $\Omega=32\pi$ & $\Omega=64\pi$ & $\Omega=128\pi$ & $\Omega=256\pi$ & $\Omega=512\pi$ & $\Omega=1024\pi$
\\ \hline 1&4.10(-2)&4.09(-2)&4.08(-2)& 4.08(-2) & 4.08(-2) & 4.07(-2) & 4.07(-2) &4.07(-2)\\ 2&***&8.08(-3)&7.89(-3)&7.81(-3)& 7.77(-3) & 7.76(-3) & 7.75(-3)  & 7.75(-3)\\ 4&***&***&1.78(-3)&1.70(-3)& 1.67(-3) & 1.65(-3) & 1.64(-3) &1.64(-3)\\ 8&***&***&***&4.29(-4)& 4.08(-4) & 3.99(-4) & 3.99(-4) & 4.06(-4)\\ 16&***&***&***&***& 1.06(-4) & 1.01(-4) & 1.04(-4) & 1.07(-4)\\32&***&***&***&***& ***& 2.62(-5) & 2.50(-5) &2.68(-5)\\ 64&***&***&***&***& *** & *** & 6.53(-6) &6.47(-6)\\ 128&***&***&***&***& *** & *** & *** &1.63(-6)\\
\hline
\end{tabular}}
\end{center}
\label{tab5}
\end{table}

We finally consider the application of SAM to problems that are not of the form \eqref{specficsys}.
 The number of variants that may arise is very high and we restrict  the attention to reporting numerical
 results for a case study. The corresponding analysis may be carried out by adapting the proofs given in
  the preceding sections.

We study again the toggle switch problem, but now in an  alternative asymptotic regime. The system is given by
\begin{eqnarray}\label{geneproblem}
\frac{dx_1}{dt}&=&\frac{\alpha}{1+x_2^{\beta}}-x_1(t-\tau)+A\sin(\omega t)+\hat{B}\Omega\sin(\Omega t),\\ \nonumber
\frac{dx_2}{dt}&=&\frac{\alpha}{1+x_1^{\beta}}-x_2(t-\tau),
\end{eqnarray}
where \(\hat B\) is a constant and the other symbols are as before.
As \(\Omega\rightarrow \infty\), the variable \(x_1\) undergoes oscillations of frequency \(\Omega\) and \(\mathcal{O}(1)\) amplitude, which, for \(\Omega\) large, makes the direct numerical integration of the system more expensive  than that of
\eqref{toggle} (the amplitude there is \(\mathcal{O}(1/\Omega)\)).
For an analytic treatment, we begin by performing, for \(t\geq 0\), the preliminary stroboscopic change of variables
\begin{eqnarray*}
x_1&=&X_1+\hat{B}(1-\cos(\Omega t)),\\
x_2&=&X_2,
\end{eqnarray*}
which differs from the identity in \(\mathcal{O}(1)\) quantities. This leads to:
\begin{eqnarray}\label{geneproblem2}
\frac{dX_1}{dt}&=&\frac{\alpha}{1+X_2^{\beta}}-X_1(t-\tau)-\hat B\,1_{\{t\geq \tau\}}+A\sin(\omega t),\\ \nonumber
\frac{dX_2}{dt}&=&\frac{\alpha}{1+\big(X_1+\hat{B}(1-\cos(\Omega t))\big)^\beta}-X_2(t-\tau).
\end{eqnarray}
The highly oscillatory forcing has been reduced from \(\mathcal{O}(\Omega)\) to \(\mathcal{O}(1)\) and, in principle, it is possible to average \eqref{geneproblem2} by the techniques used to deal with \eqref{specficsys}. Unfortunately, finding the required Fourier coefficients in closed form does not appear to be possible in general. In the particular case
where \(\beta = 2\), the \(0\)-Fourier coefficient may be found by evaluating the relevant integral \eqref{eq:integral} with the help of the residue theorem. This leads to the averaged system of the form \(\dot X = f_0\) explicitly given by
\begin{eqnarray}\label{averagesys}
\frac{dX_1}{dt}&=&\frac{\alpha}{1+X_2^{2}}-X_1(t-\tau)-\hat{B}\,1_{\{t\geq \tau\}}+A\sin(\omega t),\\ \nonumber
\frac{dX_2}{dt}&=&\alpha\frac{\sqrt{-\frac{M}{2}+\frac{\sqrt{N}}{2}}+(X_1+\hat{B})\sqrt{\frac{M}{2}+\frac{\sqrt{N}}{2}}}
{\left(\frac{M}{2}+\frac{\sqrt{N}}{2}\right)^2+(X_1+\hat{B})^2+\sqrt{N}}-X_2(t-\tau),
\end{eqnarray}
with $M=X_1^2+2\hat{B}X_1-1$,
$N=M^2+4(X_1+\hat{B})^2$, whose solutions
approximate \(x\) with errors of size at least
\(\mathcal{O}(1/\Omega)\) at stroboscopic times.
However, since \eqref{geneproblem2} is even in
\(\Omega\), in this particular case the errors
are actually \(\mathcal{O}(1/\Omega^2)\).

Table~\ref{tab5} presents the errors in the SAM solution measured with respect to the solution of
\eqref{averagesys}. The  experiments have \(\hat B = 0.1\); all other details are as in the toggle switch
simulations in the preceding section. The table shows that the performance of SAM is very similar to that
encountered in problem \eqref{toggle}. We also measured errors with respect to the oscillatory solution and
found that they are very close to those reported here, i.e.\ the situation is similar to that seen when
comparing Tables~\ref{tab2} and \ref{tab3}.

Finally we mention that, for the choice of
constants considered here, the integration (in
the interval \(0\leq t\leq 2\)) of the
oscillatory problem for  \(\Omega=1024\pi\) with
dde23 in a laptop computer took more than 9,000
seconds. The corresponding SAM solution with the
smallest value of \(H\) took approximately one
second. Since as pointed out before, the study of
vibrational resonance requires integrations in
time intervals two orders of magnitude larger
than \(0\leq t\leq 2\), for many choices of the
values of the constants that appear in the model,
it is clear that a direct numerical integration of the
oscillatory problem is not feasible for large
values of \(\Omega\).
\bigskip

{\bf Acknowledgements.}  J.M. Sanz-Serna has been supported by projects MTM2013-46553-C3-1-P from Ministerio
de Econom\'{\i}a y Comercio, and MTM2016-77660-P(AEI/\-FEDER, UE) from Ministerio de Eco\-nom\'{\i}a,
Industria  y Competitividad, Spain. Beibei Zhu has been supported by the National Natural Science Foundation
of China (Grant No. 11371357 and No. 11771438). She is grateful to Universidad Carlos III de Madrid for
hosting the stay in Spain that made this work possible and to the Chinese Scholarship Council for providing
the necessary funds. The authors are thankful to M. A. F. Sanju\'{a}n and A. Daza for bringing to their
attention the vibrational resonance phenomenon, the toggle switch problem and other highly-oscillatory systems
with delay.

\bibliography{IMANUM-refs}


\end{document}